\documentclass[a4paper]{amsart}

\usepackage{amsfonts}
\usepackage{amssymb}
\usepackage{amsmath}
\usepackage{hyperref}
\usepackage{mathrsfs}
\usepackage{centernot}
\usepackage{mathdots}
\usepackage{stmaryrd}
\usepackage[all]{xy}

\usepackage{mathtools}

\usepackage{color}

\newtheorem{thm}{Theorem}[section]
\newtheorem{lem}[thm]{Lemma}
\newtheorem{prp}[thm]{Proposition}
\newtheorem{cor}[thm]{Corollary}

\newtheoremstyle{roman} 
    {8.0pt plus 2.0pt minus 4.0pt}                    
    {8.0pt plus 2.0pt minus 4.0pt}                    
    {\normalfont}                
    {}                           
    {\bfseries}                  
    {.}                          
    {5pt plus 1pt minus 1pt}     
    {}  

\theoremstyle{roman}

\newtheorem{example}[thm]{Example}
\newtheorem{remark}[thm]{Remark}

\theoremstyle{plain}

\newcommand{\rem}[1]{}

\newcommand{\N}{\mathbb{N}}
\newcommand{\Q}{\mathbb{Q}}
\newcommand{\R}{\mathbb{R}}
\newcommand{\Z}{\mathbb{Z}}

\newcommand{\calP}{{\mathcal{P}}}

\newcommand{\catC}{{\mathscr{C}}}

\newcommand{\catZ}{{\mathscr{Z}}}

\newcommand{\veps}{\varepsilon}

\newcommand{\idealof}{\unlhd} 

\newcommand{\suchthat}{\,:\,}
\newcommand{\where}{\,|\,}

\newcommand{\quo}[1]{\overline{#1}}

\DeclareMathOperator{\Cent}{Cent} %
\DeclareMathOperator{\Char}{char} %
\DeclareMathOperator{\End}{End} %
\DeclareMathOperator{\Hom}{Hom} %
\DeclareMathOperator{\id}{id} %
\DeclareMathOperator{\im}{im} %
\DeclareMathOperator{\Jac}{Jac} %

\newcommand{\op}{\mathrm{op}} %
\DeclareMathOperator{\Span}{span} %
\DeclareMathOperator{\Tr}{Tr} %
\DeclareMathOperator{\Trd}{Trd} %

\newcommand{\nGL}[2]{\mathrm{GL}_{#2}({#1})}

\newcommand{\nMat}[2]{\mathrm{M}_{#2}(#1)}

\newcommand{\trans}{{\mathrm{t}}}

\newcommand{\units}[1]{{#1^\times}}

\newcommand{\uddots}{\reflectbox{$\ddots$}}

\usepackage{enumitem} 

\usepackage[foot]{amsaddr}  

\DeclareMathOperator{\sgn}{sgn}

\newcommand{\tAr}[2]{\mathrm{A\tilde{r}}_{#2}({#1})}

\newcommand{\Sys}{{\mathrm{Sys}}}

\newcommand{\Herm}[2][]{\mathrm{UH}^{#1}({#2})}

\newcommand{\Cad}{{\mathrm{Cl}}}

\newtheorem{introthm}{Theorem}

\newtheorem{introcor}[introthm]{Corollary}

\newtheorem{introque}[introthm]{Question}

\title{Pfister's Local--Global Principle and
Systems of Quadratic Forms}

\author{Uriya A.\ First$^*$}
\address{$^*$University of Haifa}
\email{uriya.first@gmail.com}


\begin{document}

\maketitle

\begin{abstract}
Let $q$ be a unimodular quadratic form over a field $K$.
Pfister's famous local--global principle asserts
that $q$ represents a torsion class in the Witt group of $K$
if and only if it has signature $0$, and that in this
case, the order of Witt class of $q$ is a power of $2$.
We give two analogues of this result to systems of quadratic forms,
the second of which applying only to nonsingular pairs.
We also prove a counterpart of   Pfister's theorem for finite-dimensional
$K$-algebras with involution, generalizing a result of Lewis and Unger.
\end{abstract}

\section*{Introduction}

Let $K$ be a field of characteristic different from $2$
and let $(V,q)$ be a unimodular (i.e.\ nondegenerate) quadratic space over $K$.
We write $n\times q$ for the quadratic form $(v_1,\dots,v_n)\mapsto \sum_i q(v_i):V^n\to K$.

Pfister's celebrated local-global principle 
(see \cite[Theorem~2.7.3]{Scharlau_1985_quadratic_and_hermitian_forms}, for instance) 
states that
there exists $n\in\N$ such that $n\times q$ is hyperbolic if and only if
the signature of $q$ (relative to all orderings of the field $K$) is $0$,
and that in this case, $n$ can be taken to be a power of $2$.
This work is concerned with analogues of this result to systems
of quadratic forms, and in particular to pairs of forms.

To that end, we say that a system of quadratic forms $\{q_i\}_{i\in I}$
on a $K$-vector space $V$ is \emph{hyperbolic} if 
$V$ is the direct sum of two $K$-subspaces on which each of
the forms $q_i$ vanishes.
This is one of several possible notions
for a ``trivial'' system of quadratic forms listed
by Pfister in   \cite[p.~133]{Pfister_1995_quadratic_forms_and_applications};
it is the most suitable for our purposes as it implies
that every form in the $K$-span of the system has signature $0$ (Proposition~\ref{PR:Pfister-for-non-unimod}). 
This definition   also appeared   in \cite[\S4]{Bayer_2014_hermitian_categories},
and if it is applied to  single non-unimodular quadratic forms, then Pfiter's local-global principle
still holds as stated, see Proposition~\ref{PR:Pfister-for-non-unimod}.

It is tempting to hope that if every quadratic form in $  \Span_K\{q_i\where i\in I\}$
has signature $0$, then $n\times \{q_i\}_{i\in I}$ is hyperbolic for some $n$.
However, as we  demonstrate  in Section~\ref{sec:examples}, this is already false for pairs of forms.
Therefore, one cannot expect Pfister's local-global principle to generalize
naively to systems of forms, and indeed,
the analogues that we shall give here will  take a more sophisticated
form.

\medskip

To phrase our  results, let $A=A(\{q_i\}_{i\in I})$ 
denote the $K$-subalgebra of $\End(V)\times\End(V)^\op$ consisting
of pairs $(\phi,\psi^\op)$ satisfying
	\[
	q_i(\psi x,y)=q_i(x,\phi y)
	\]
for all $x,y\in V$ and $i\in I$, and let $\sigma:A\to A$ denote
the involution given by
\[
	(\phi,\psi^\op)^\sigma=(\psi,\phi^\op).
\]
This construction had been utilized by many authors, e.g.,
Bayer-Fluckiger  
\cite[\S1.1]{Bayer_1987_Hasse_principle_for_systems_of_forms}
and Wilson \cite[\S4.3]{Wilson_2009_decomposing_p_groups}.
Following 
the latter source, we call $A$ the 
\emph{$K$-algebra of adjoints} of $\{q_i\}_{i\in I}$ 
and $\sigma$   its \emph{canonical involution}.

Denote the \emph{involution-trace} quadratic form
$x\mapsto \Tr_{A/K}(x^\sigma x):A\to K$ by $q_{A,\sigma}$. Our first main results is:

\begin{introthm}
	\label{TH:main-I}
	In the previous notation, the following conditions are equivalent:
	\begin{enumerate}[label=(\alph*)]
			\item $n\times \{q_i\}_{i\in I}$ is hyperbolic for some $n\in\N$.
			\item $\sgn q_{A,\sigma}=0$.
	\end{enumerate}
	When these conditions hold, the minimal   $n$ for which
	(a) holds is a power of $2$.
\end{introthm}

This is   a generalization of Pfister's local-global principle because, when $\{q_i\}_{i\in I}$
consists of a single form $q$, we have $\sgn q_{A,\sigma}=(\sgn q)^2$ 
(Proposition~\ref{PR:adj-of-single-form}(i)).

In the course of proving this result,
we
show that a finite-dimensional $K$-algebra with involution $(A,\sigma)$ admits $n\in\N$
such that  $(A,\sigma)\otimes_K (\nMat{K}{n},\trans)$ is \emph{hyperbolic}
(see Section~\ref{sec:algs-with-inv}) if and if $\sgn q_{A,\sigma}=0$
(Theorem~\ref{TH:trace-general}).
This generalizes
a theorem of  Lewis and Unger \cite[Theorem~3.2]{Lewis_2003_local_global_princ_alg_with_inv},
who   established the case where $(A,\sigma)$ is central simple over $K$.

\medskip

Our second generalization of Pfister's local-global principle
applies only to   pairs of quadratic forms, 
but is  in the spirit of the   naive statement
we disqualified above.

Letting $(V,\{q_i\}_{i\in I})$ and $(A,\sigma)$ be as before, it can happen
that adjoining a quadratic form $q:V\to K$ to 
the system $\{q_i\}_{i\in I}$ will not  change the algebra of adjoints $A$. For
example, this always the case if $q\in\Span_K \{q_i\where i\in I\}$.
We denote by $\Cad\{q_i\where i\in I\}$ the $K$-vector space
of all such forms; this construction was introduced to us
by James Wilson, who also observed that $\Cad(-)$
is a closure operator.

By virtue of Theorem~\ref{TH:main-I}, if some $q\in \Cad\{q_i\where i\in I\}$
has nonzero signature, then there cannot exist an
$n\in\N$ such that $n\times \{q_i\}_{i\in I}$ is hyperbolic. 
Our second main result asserts the converse of this statement
for \emph{nonsingular} pairs of quadratic
forms,   provided $K$ is a number field or real closed.
Here, a system  of quadratic forms $\{q_i\}_{i\in I}$
is called nonsingular 
if 
there is a $K$-field $L$ such that $\Span_L\{(q_i)_L\where i\in I\}$
contains a unimodular quadratic form over $L$ (one can take $L=K$ if $K$ is infinite). 

\begin{introthm}\label{TH:main-II}
	Suppose that $K$ is a number field or a real closed field
	and let $\{q_i\}_{i=1,2}$ be a nonsingular pair of quadratic forms
	on a $K$-vector space $V$.
	Then the following conditions are equivalent:
	\begin{enumerate}[label=(\alph*)]
			\item $n\times \{q_i\}_{i=1,2}$ is hyperbolic for some $n\in\N$.
			\item Every $q\in \Cad\{q_1,q_2\}$ has
			signature $0$.
	\end{enumerate}
\end{introthm}

If $q_1=q_2$, then $\Cad\{q_1,q_2\}=K q_1$ (Proposition~\ref{PR:adj-of-single-form}),
so Theorem~\ref{TH:main-II} also generalizes Pfister's local-global principle.

In fact, Theorem~\ref{TH:main-II}
holds for all fields $K$
satisfying a certain condition
(see Section~\ref{sec:pairs}),
which we believe to hold for all fields. We  conjecture
that the nonsingularity
assumption can be removed as well.

We further note that Theorem~\ref{TH:main-I} implies
that there exists $n\in\N$
such that $n\times \{q_i\}_{i\in I}$ is hyperbolic if and only if
the same statement holds after base-changing to the real closure of $K$
relative to each of its ordering. Thus, writing $K_P$
for the real closure of $K$ relative to an ordering $P$,
we have the following corollary,
which holds over any field.

\begin{introcor}\label{CR:main-III}
	Let $\{q_i\}_{i=1,2}$ be a nonsingular pair of 
	quadratic forms on a $K$-vector space $V$.
	Then the following conditions are equivalent:
	\begin{enumerate}[label=(\alph*)]
		\item $n\times \{q_i\}_{i=1,2}$ is hyperbolic for some $n\in\N$.
		\item[(b$'$)]
		For every ordering $P$ of $K$
		and every $q\in \Cad\{(q_1)_{K_P},(q_2)_{K_P}\}$,
		we have $\sgn q=0$. 
	\end{enumerate}		
\end{introcor}

Finally, we ask:

\begin{introque}\label{QE:main}
	Do Theorem~\ref{TH:main-II} and
	Corollary~\ref{CR:main-III} apply to systems
	consisting of more than $2$ quadratic forms?
\end{introque}

The paper is organized as follows:
Section \ref{sec:preliminaries} is preliminary and recalls relevant
definitions and facts.
In 
Section~\ref{sec:examples}, we give  nontrivial examples of vector spaces
of quadratic forms consisting of forms with signature $0$, and demonstrate
our main results on them.
Section~\ref{sec:algs-with-inv} concerns with generalizing
Lewis and Unger's theorem stated above. This is used in
Section~\ref{sec:systems-main} to prove Theorem~\ref{TH:main-I}.
The remaining two sections concern with proving Theorem~\ref{TH:main-II}:
Section~\ref{sec:herm-cat} recalls necessary facts about hermitian categories,
and the proof itself is given in Section~\ref{sec:pairs}.

\medskip

We are grateful to Eva Bayer-Fluckiger and David Leep
for several useful conversations. We
also thank  James Wilson for introducing to us
the notion of closure of sets of quadratic forms used in 
Theorem~\ref{TH:main-II}.

\section{Preliminaries}
\label{sec:preliminaries}

Throughout this paper, $K$ denotes a field of characteristic not $2$.
All $K$-vector spaces and $K$-algebras are assumed
to be finite-dimensional.

We refer the reader to \cite{Scharlau_1985_quadratic_and_hermitian_forms} for
necessary definitions concerning quadratic, bilinear   and hermitian forms.

Our assumptions on the characteristic of $K$
allows us no to   distinguish between quadratic and bilinear forms,
and we will use
the same letter to  denote  a quadratic form $q:V\to K$ and its associated
bilinear form 
$(x,y)\mapsto \frac{1}{2}(q(x+y)-q(x)-q(y)):V\times V\to K$.
As usual, 
given $\alpha_1,\dots,\alpha_n\in K$,
the diagonal quadratic $(x_1,\dots,x_n)\mapsto \sum_i \alpha_ix_i^2:K^n\to K$
is denoted $\langle \alpha_1,\dots,\alpha_n\rangle$.
We do not require quadratic forms to be unimodular
(i.e.\ nondegenerate).

If $L$ is a $K$-field, $U$ and $V$ are   $K$-vector spaces
and $\phi\in\Hom_K(U,V)$, then we write $U_L=U\otimes_KL$
and $\phi_L=\phi\otimes_K\id_L:U_L\to V_L$.
Similar notation will be applied to algebras, quadratic forms,
involutions, etcetera.

\medskip

Recall that an ordering $P$ of $K$ is a subset of $\units{K}:=K-\{0\}$
such that $P+P\subseteq P$, $P\cdot P\subseteq P$,
$\units{K}=P\cup -P$ and $P\cap -P=\emptyset $. In this case, given $\alpha,\beta \in K$, we write
$\alpha<_P \beta$ if $\beta-\alpha\in P$,
and set 
\[\sgn_P (\alpha)=\left\{\begin{array}{rl}
1 & \alpha\in P \\
0 & \alpha=0 \\
-1 & \alpha\in -P
\end{array}\right. . \] 
If $P$ is clear form the context, we shall
suppress it and write $\alpha<\beta$, resp.\ $\sgn(\alpha)$. 
The real closure of $K$ relative to $P$ is denoted $K_P$.

Let $P$ be an ordering of $K$ and let $(V,q)$  be a quadratic space over $K$.
Recall that $q$ is called positive (resp.\ negative) definite relative
to $P$
if $q(v)>0$ (resp.\ $q(v)<0$) for all $v\in V-\{0\}$.
The $P$-signature of $q$, denoted $\sgn_P q$,
is largest possible dimension of a subspace on which $q$
is positive definite minus the largest possible
dimension of a subspace on which $q$ is negative definite.
Note that this definition also makes sense
for  non-unimodular forms, and we have
\[
\sgn_P\langle \alpha_1,\dots,\alpha_n\rangle=\sum_i\sgn_P(\alpha_i).
\]
Let $\Theta$ denote the set of all orderings of $K$.
The \emph{total signature} of $q$, denoted
$\sgn q $, is the function $\Theta\to \Z$
mapping $P$ to $\sgn_P(q)$.

\medskip

Let $(V,q)$ be quadratic space.
Recall from the introduction
that $q$ is called \emph{hyperbolic}
if there exist subspaces $U,U'\subseteq V$
such that $V=U\oplus U'$ and $q$ vanishes on $U$
and $U'$. This agrees with the usual definition
of hyperbolic quadratic forms when $q$ is unimodular.
The following proposition summarizes some properties
of hyperbolic quadratic forms in the non-unimodular case.
Notably, such forms have signature $0$.

	\begin{prp}\label{PR:Pfister-for-non-unimod}
		Let $(V,q)$ be a quadratic space, possibly non-unimodular.  Then:
		\begin{enumerate}[label=(\roman*)]
			\item   $q$ is hyperbolic 
			if and only if $q\cong q'\oplus \langle 0,\dots,0\rangle$
			for some hyperbolic unimodular quadratic form $q'$.
			\item Pfister's local-global principle holds for $(V,q)$: There
			exists $n\in\N$ such that $n\times q$ is hyperbolic
			if and only if $\sgn q=0$, and in this case,
			the smallest such $n$ is a power of $2$.
		\end{enumerate}
	\end{prp}
	
	\begin{proof}
		(i) The ``if'' part is clear.
		For the ``only if'' part, write $V=U\oplus W$ so
		that $q$ vanishes on both $U$ and $W$.
		Let $R$ denote the radical of $q$, let $\quo{V}=V/R$
		and let $q':\quo{V}\to K$
		denote the quadratic form given by $q'(x+R)=q(x)$
		(this is well-defined because $R=V^\perp$).
		Then $q'$ is unimodular and vanishes
		on $\quo{U}:=(R+U)/R$ and $\quo{W}:=(R+W)/R$, hence $q'$ is hyperbolic.
		Since $q\cong q'\oplus \langle 0,\dots,0\rangle$, we are done.

		(ii) This follows from (i) and Pfister's local-global principle. 
	\end{proof}

	Let $V$ be a $K$-vector space and let $I$ be a set.
	By an $I$-indexed system of quadratic forms on $V$
	we mean a collection $\{q_i\}_{i\in I}$ consisting of
	quadratic forms on $V$. We also say that $(V,\{q_i\}_{i\in I})$
	is a vector space with a system of quadratic forms. 
	Given $n\in \N$, we write $n\times\{q_i\}_{i\in I}=\{n\times q_i\}_{i\in I}$,
	which is  an $I$-indexed system of forms on $V^n$.
	Recall that $\{q_i\}_{i\in I}$ is called
	\emph{nonsingular} if there is a $K$-field
	$L$ such that $\Span_L\{(q_i)_L\where i\in I\}$ contains a unimodular form.
	
	\begin{lem}\label{LM:nonsingular-systems}
		Let $ \{q_i\}_{i\in I} $ be a nonsingular system of quadratic
		forms on a $K$-vector space $V$.
		If $K$ is infinite, then $\Span_K\{q_i\where i\in I\}$
		contains a unimodular form.
	\end{lem}

	\begin{proof}
		Let $L$ be a field such that $\Span_L\{(q_i)_L\where i\in I\}$ contains a unimodular form.
		Then there exist $t\in \N$, $q_1,\dots,q_t\in \{q_i\where i\in I\}$
		and $\alpha_1,\dots,\alpha_t\in L$
		such that $ \sum_{i=1}^t\alpha_i (q_i)_L $
		is a unimodular quadratic form over $L$. 
		Let    $B$ be a basis to $V$
		and let $f \in K[x_1,\dots,x_t]$
		denote the determinant of    $\sum_i x_iq_i$ relative to $B$.
		Then $f(\alpha_1,\dots,\alpha_t)\neq 0$,
		hence $f\neq 0$. Since $K$ is infinite, there are $\beta_1,\dots,\beta_t\in K$
		such that $f(\beta_1,\dots,\beta_t)\neq 0$.
		Then $\sum_i\beta_iq_i\in \Span_K\{q_i\where i\in I\}$ is unimodular.
	\end{proof}

\section{Examples}
\label{sec:examples}

	Before setting to prove Theorems~\ref{TH:main-I}
	and~\ref{TH:main-II}, 
	we first exhibit nontrivial examples of
	systems of quadratic forms with $K$-span consisting
	of signature-$0$ forms.
	In particular, we shall see that the dimension of the $K$-span
	can be arbitrary large, even when the system
	is not hyperbolic, and that such systems $\{q_i\}_{i\in I}$ 
	may fail to admit $n\in\N$ such that $n\times\{q_i\}_{i\in I}$
	is hyperbolic.

	For the sake of brevity, quadratic forms on $K^n$
	will be given simply as their Gram matrix relative to the standard basis.
	In this setting, the ring of adjoints of a system of quadratic forms $\{q_i\}_{i\in I}$
	on $K^n$ is the collection of pairs $(\phi,\psi^\op)\in \End(K^n)\times \End(K^n)^\op=\nMat{K}{n}\times\nMat{K}{n}^\op$
	satisfying	
	\[
	\psi^\trans q_i=q_i\phi  \qquad \forall 
	i\in I.
	\]
	Moreover, a quadratic form $q:K^n\to K$ (viewed as a symmetric matrix)
	lies in $\Cad\{q_i\where i\in I\}$ if and only if
	\[
	\psi^\trans q =q\phi \qquad\forall(\phi,\psi^\op)\in A(\{q_i\}_{i\in I}).
	\]

	\begin{example}\label{EX:good}
		Let $\{q_i\}_{i=1}^{n^2+1}$ be
		a basis to  the space of quadratic forms on $K^{2n}$ 
		taking the form
	\[
	\left[
	\begin{matrix}
	\alpha I_n & a \\
	a^\trans & -\alpha I_n
	\end{matrix}
	\right],
	\]
	where $I_n$ denotes the $n\times n$ identity
	matrix, $a\in \nMat{K}{n}$ and $\alpha\in K$.
	To see that every form in $Q:=\Span_K\{q_i\}_{i=1}^{n^2+1}$
	has signature $0$, let $P$ be an ordering of $K$
	and consider
	$U=K^n\times \{0\}^n$ and $W=\{0\}^n\times K^n$.
	Then for every $q\in Q$, exactly one of the following holds:
	\begin{itemize}
		\item $q|_{U }$ is positive definite and $q|_{W}$ is negative definite relative to $P$;
		\item $q|_{U }$ is negative definite and $q|_{W}$ is positive definite relative to $P$;
		\item $q$ vanishes on both $U$ and $W$.
	\end{itemize} 
	Each of these   possibilities implies $\mathrm{sgn}_P q=0$, so $Q$ is an $(n^2+1)$-dimensional
	space consisting of $0$-signature forms.

	The system $\{q_i\}_{i=1}^{n^2+1}$ is not hyperbolic because there
	is no nonzero vector which is annihilated by
	all forms in the system.
	However, $2\times\{q_i\}_{i=1}^{n^2+1}$   is hyperbolic. Indeed,   
	$2\times q_i$ vanishes
	on
	\begin{align*}
	V_1&:=\{(u,v,-v,u)\where u,v\in K^n\}
	\qquad
	\text{and}
	\qquad
	V_2&:=\{(u,v,v,-u)\where u,v\in K^n\}
	\end{align*}
	for all $i$,
	and $K^{4n}=V_1\oplus V_2$
	(here we view $K^{4n}$ as
	$K^n\times K^n\times K^n\times K^n$).
	\end{example}

		\begin{example}
		Take $K=\R$ and
		let $\{q_1,q_2\}$ be 
		a basis to the space $Q$ of quadratic forms on $K^4$
		of the form
		\[
		\begin{bmatrix}
		-\alpha & & & \\
		& & & \alpha \\
		& & \alpha & \beta \\
		& \alpha & \beta &  
		\end{bmatrix}.
		\]
		It is easy to see that every $q\in Q$ is hyperbolic,
		and thus has  signature $0$. However, there is no $n\in \N$
		such that $n\times \{q_i\}_{i=1,2}$ is hyperbolic.
		Indeed,  straightforward
		computation   shows that $A=A(\{q_1,q_2\})$ is the $K$-subalgebra
		of $\nMat{K}{4}\times\nMat{K}{4}^\op$ consisting of pairs of the form
		\[
		\left(
		\left[
		\begin{matrix}
		x & & & y \\
		z & \alpha & \beta & \gamma \\
		& & \alpha & \beta \\
		& & & \alpha
		\end{matrix}
		\right],
		\left[
		\begin{matrix}
		x &   & & -z  \\
		-y & \alpha & \beta & \gamma \\
		& & \alpha & \beta \\
		& & & \alpha
		\end{matrix}
		\right]^\op
		\right).
		\]
		This in turn implies
		that   $\Cad \{q_1,q_2\} $ 
		is
		the $3$-dimensional space of quadratic forms 
		of the form		
		\[
		\begin{bmatrix}
		-\alpha & & & \\
		& & & \alpha \\
		& & \alpha & \beta \\
		& \alpha & \beta & \gamma  
		\end{bmatrix},
		\]
		and this space  contains forms of nonzero signature
		(take $\alpha=\beta=0$ and $\gamma\neq 0$).
		
		Alternatively, one can check   that $\sgn q_{A,\sigma}=2$,
		and reach the same conclusion using Theorem~\ref{TH:main-I}.
	\end{example}

	\begin{example}
		Let $K=\Q$ and let $\{q_1,q_2\}$
		be a basis to 
		the $2$-dimensional $K$-vector space $Q$ consisting of quadratic forms on $K^4$
		of the form
		\[
		\begin{bmatrix}
		 & & \alpha & \beta \\
		& & 2\beta & \alpha \\
		\alpha & 2\beta  & \beta &  \\
		 \beta & \alpha &  
		\end{bmatrix}.
		\]
		Then every $q\in Q$ is either
		the zero form or a hyperbolic unimodular quadratic form,
		hence $Q$ consists entirely of forms with signature $0$.
		However,  there exist forms in $Q\otimes_K\R$ of nonzero signature (e.g.\ take $\alpha=2$ and $\beta=\sqrt{2}$),
		so   there is no $n\in \N$ such that $n\times \{q_i\}_{i=1,2}$ is hyperbolic.
		Theorem~\ref{TH:main-II} guarantees that we can also find
		$q\in \Cad \{q_1,q_2\} $ with nonzero signature,
		and indeed, one can check
		that the diagonal form
		$\langle 0,0,2,1\rangle$ is such an example.
	\end{example}

	We finish with a general method
	for  producing high-dimensional vector spaces of quadratic forms consisting
	of forms with signature $0$. Small-scale
	experiments suggest that applying
	it with a ``generic''
	choice of parameters 
	will
	result in a system having 
	a form of nonzero signature in its closure.

\begin{example}\label{EX:cnt}
	Let $n\in\N$ and let $S$ and $E$ be $K$-subspaces of $\nMat{K}{n}$
	such that $S$ consists of symmetric matrices representing quadratic forms of
	signature $0$, and any nonzero matrix in $E$ is invertible.
	Let
	\[
	Q=Q(S,E):=
	\left\{
	\left[
	\begin{matrix}
	0 & e \\
	e^{\trans} & s
	\end{matrix}
	\right]
	\,\big|\,
	e\in E,s\in S\right\}
	\subseteq\nMat{K}{2n}.
	\]
	We claim that the signature of any quadratic form in $Q$
	is $0$. Indeed, if $q=[\begin{smallmatrix}
	0 & e \\
	e^{\trans} & s
	\end{smallmatrix}]\in Q$, 
	then $e$ is either invertible or $0$. In the first case,
	$q$ is unimodular of dimension $2n$ and admits a totally isotropic subspace
	of dimension $n$, so  it is hyperbolic and has signature $0$,
	whereas
	in the second case, $\mathrm{sgn}(q)=\mathrm{sgn}(s)=0$.

	In the case $K=\R$, the largest possible dimension of $E$
	was determined by Adams \cite{Adams_1962_vector_fields_spheres} and equals 
	to 
	the Hurwitz--Radon number $\rho(n)$ given
	by $8a+2^b$ if $n=2^{4a+b}c$ with  $0\leq a$, $0\leq b\leq 3$ and   $c$ odd.
\end{example}

\section{Algebras with Involution and Involution-Trace Forms}
\label{sec:algs-with-inv}

	By a $K$-algebra with involution we mean a pair $(A,\sigma)$
	such that $A$ is a $K$-algebra and $\sigma$ is a $K$-involution.
	A $K$-algebra with involution  is  \emph{simple}  
	if it has no nonzero proper ideals stable under its involution.
	In this case,
	$\Cent(A)^{\{\sigma\}}:=\{a\in \Cent(A)\suchthat a^\sigma=a\}$ is a field.
	Recall that the involution-trace form of a $K$-algebra with involution $(A,\sigma)$
	is the quadratic form 
	$q_{A,\sigma}:A\to K$
	given by 
	\[q_{A,\sigma}(x)=\Tr_{A/K}(x^\sigma x).\] 
	It is not
	unimodular in general.
	
	Following \cite{Knus_1998_book_of_involutions},
	we say that $(A,\sigma)$ is a central simple $K$-algebra with involution
	if $(A,\sigma)$ is simple and $\Cent(A)^{\{\sigma\}}=K$.
	We  alert the reader that in this case, it is  common to
	define the involution trace form of $(A,\sigma)$  using the \emph{reduced} trace
	$\Trd_{A/\Cent(A)}$ instead of the trace; 
	see \cite{Lewis_1993_signature_of_involution} and \cite{Queguiner_1995_signature_of_involution}.
	
	For every $n\in\N$,  write $n\times (A,\sigma)=(\nMat{A}{n},n\times \sigma)$,
	where $n\times \sigma$ is
	the   involution 
	$(a_{ij})\mapsto (a_{ji}^\sigma):\nMat{A}{n}\to \nMat{A}{n}$.
	Then   $n\times (A,\sigma)\cong (A\otimes\sigma)\otimes_K (\nMat{K}{n},\trans)$,
	where $\trans$ denotes the matrix transpose,
	and $q_{n\times (A,\sigma)}\cong n^2\times q_{A,\sigma}$.

\medskip

	We say that    $(A,\sigma)$, or just $\sigma$,
	is   \emph{hyperbolic} if there exists
	an idempotent $e\in A$ such that $e^\sigma+e=1$. 
	The relation to hyperbolic hermitian forms 
	is expressed in the following proposition.
	
	\begin{prp}\label{PR:hyp-inv-vs-hyp-form}
		Let $(A,\sigma)$ be a $K$-algebra
		with involution, let $n\in\N$
		and  let $f_n:A^n\times A^n\to A$
		denote the $1$-hermitian form
		over $(A,\sigma)$ given by $f_n((x_i)_{i=1}^n,(y_i)_{i=1}^n)=\sum_i x_i^\sigma y_i$.
		Then $n\times  \sigma  $ is hyperbolic if and only if $f_n$
		is hyperbolic.
	\end{prp}
	
	\begin{proof}
		View the elements of the right $A$-module $A^n$ as column vectors
		and identify $\End_A(A^n)$ with $\nMat{A}{n}$.
		Writing $\tau=n\times \sigma$, it is easy 
		to see that $f_n(a  x,y)=f_n(x,a^\tau y)$ for all $a\in\nMat{A}{n}$
		and $x,y\in A^n$.
		
		Suppose that there exists an idempotent  $e\in\nMat{A}{n}$ such that
		$e^\tau+e=1$. Then $f_n(ex,ey)=f_n(x,e^\tau ey)=f_n(x,0)=0$,
		and similarly $f_n((1-e)x,(1-e)y))=0$. Since $A^n=\im(e)+\im(1-e)$,
		it follows that $f_n$ is hyperbolic.
		
		Conversely, suppose that there exist $A$-submodules $U,V\subseteq A$
		such that $A^n=U\oplus V$ and $f(U,U)=f(V,V)=0$,
		and let $e=\id_U\oplus 0_V\in \nMat{A}{n}$.
		Then for all $x,y\in A^n$, we have 
		$f_n(x,e^\tau y)=f_n(e x,y)=f_n(ex,ey+(1-e)y)=f_n(e x,(1-e)y)=
		f_n(ex+(1-e)x,(1-e)y)=f_n(x,(1-e)y)$. 
		Since $f_n$ is unimodular, $e^\tau=1-e$ and $\tau$ is hyperbolic.
	\end{proof}
	
	We record the following   corollary:
	
	\begin{cor}
		\label{CR:gcd-for-algebras}
		Let $(A,\sigma)$ be a $K$-algebra
		with involution and let $n,m\in\N$.
		If $n\times \sigma$
		and $m\times \sigma$
		are hyperbolic, then so is $\gcd(n,m)\times \sigma$.
	\end{cor}

	\begin{proof}
		By Proposition~\ref{PR:hyp-inv-vs-hyp-form},
		we need to show that $f_{\gcd(n,m)}=\gcd(n,m)\times f_1$
		is hyperbolic. By assumption, $f_n=n\times f_1$
		and $f_m=m\times f_1$ 
		represent the trivial class in the Witt group of $(A,\sigma)$,
		so $\gcd(n,m)\times f_1$ also represents
		the trivial class.
		By \cite[Proposition~5.12]{Bayer_2014_hermitian_categories} (for instance),
		this means that $f_{\gcd(n,m)}$ is hyperbolic. 
	\end{proof}

	The purpose of this section is to prove the following theorem,
	which was established by Lewis and Unger \cite[Theorem~3.2]{Lewis_2003_local_global_princ_alg_with_inv}
	for central simple $K$-algebras with involution.

	\begin{thm}\label{TH:trace-general}
		Let $(A,\sigma)$
		be   a finite-dimensional
		$K$-algebra with  involution. Then the following conditions
		are equivalent:
		\begin{enumerate}[label=(\alph*)]
			\item $n\times (A,\sigma)$ is hyperbolic 
			for some $n\in \N$.
			\item  $\sgn q_{A,\sigma}=0$.
		\end{enumerate}
		When these conditions hold, the minimal   $n$ for which
		(a) holds is a power of $2$.
	\end{thm}	
	
	Similarly to \cite{Lewis_2003_local_global_princ_alg_with_inv},
	we first establish the theorem when   $K$ real-closed,
	and then use it to prove    the general case.

	\begin{lem}\label{LM:the-real-simple-case}
		Suppose that $K$ is real closed and $(A,\sigma)$
		is a simple $K$-algebra with involution. Then:
		\begin{enumerate}[label=(\roman*)]
			\item $\sgn q_{A,\sigma}\geq 0$.
			\item $\sgn q_{A,\sigma}=0$ if and only if $2\times (A,\sigma)$
			is hyperbolic.
		\end{enumerate}
	\end{lem}
	
	\begin{proof}
		The case where $(A,\sigma)$ is \emph{central} simple over $K$
		is contained in \cite[Corollary~5.3]{Becher_2018_weakly_hyerbolic_involutions}.
		It remains to consider the case where $C:=\Cent(A)^{\{\sigma\}}$ is strictly
		larger than $K$.
		Since $K$ is real closed, $C=K[\sqrt{-1}]$ and $C$ is algebraically
		closed.
		In this case, it is well-known that there exists $n\in\N$
		such that one of the following hold:
		\begin{enumerate}[label=(\roman*)]
			\item $(A,\sigma)\cong (\nMat{C}{n},\trans)$;
			\item $(A,\sigma)\cong (\nMat{C}{2n},\mathrm{s})$
			with  $\mathrm{s}$   given by $[\begin{smallmatrix} a & b \\ c & d \end{smallmatrix}]
			\mapsto [\begin{smallmatrix} d^\trans & -b^\trans \\ -c^\trans & a^\trans \end{smallmatrix}]$
			($a,b,c,d\in\nMat{K}{n}$);
			\item $(A,\sigma)\cong (\nMat{C}{n}\times \nMat{C}{n}^\op,\mathrm{u})$ with
			$\mathrm{u}$   given by $(a,b^\op)\mapsto (b,a^\op)$.
		\end{enumerate}
		It is routine to check that in each of these cases
		$2\times (A,\sigma)$ is hyperbolic and $\sgn q_{A,\sigma}=0$.
	\end{proof}
	
	\begin{lem}\label{LM:trace-description}
		Suppose that $\Char K=0$.
		Let $(A,\sigma)$ be a $K$-algebra with involution,
		let $J $ denote its Jacobson radical,
		let $\quo{A}=A/J$ and let $\quo{\sigma}:\quo{A}\to \quo{A}$
		denote the   involution induced by $\sigma$.
		Then:
		\begin{enumerate}[label=(\roman*)]
			\item $(\quo{A},\quo{\sigma})$
			factors as a product $\prod_{i=1}^t(A_i,\sigma_i)$
			of simple algebras with involution.
			\item There exist  $\alpha_1,\dots,\alpha_t \in \Q_{>0}\subseteq K$
			such that
			\[q_{A,\sigma}\cong \alpha_1 q_{A_1,\sigma_1}\oplus\dots
			\oplus \alpha_t q_{A_t,\sigma_t} \oplus\langle 0,\dots,0\rangle .\]
		\end{enumerate}
	\end{lem}

	\begin{proof}
		(i) This follows from the fact that $\quo{A}$ is semisimple
		and $\quo{\sigma}$ permutes the central idempotents of $\quo{A}$.
		The details are left to the reader.	
	
		(ii)	
		For every
		$a\in A$, write $a^{(i)}$ for the image of $a$ under $A\to \quo{A}\to A_i$.
		If $M$ is a left $A$-module, let $a|_M$ denote  the linear
		operator $x\mapsto ax:M\to M$.
		
		Since $J$ is a left submodule of $A$,
		we have
		$\Tr_{A/K}(a)=\sum_{j\geq 0} \Tr(a|_{J^j/J^{j+1}})$  for all $a\in A$.
		Write $J^j/J^{j+1}=\prod_{i=1}^t V_{i,j}$,
		where $V_{i,j}$ is a left $A_i$-module.
		Then
		\begin{align}\label{EQ:trace-J-description}
		\Tr_{A/K}(a)=\sum_{j\geq 0}\sum_{i=1}^t \Tr(a^{(i)}|_{V_{i,j}}).
		\end{align}
		
		Choose $K$-subspaces $S_1,\dots,S_t\subseteq A$
		such that $A=J\oplus S_1\oplus\dots\oplus S_t$
		and $\quo{S_i}=A_i$ for all $i$. Then $S_iS_{i'}\subseteq J$ for   $i\neq i'$.
		By \eqref{EQ:trace-J-description}, $q_{A,\sigma}$
		vanishes on $J$, so
		$S_1,\dots,S_t,J$ are pairwise orthogonal
		relative to $q_{A,\sigma}$.
		
		We claim that for all $i\in\{1,\dots,t\}$ and $j\geq 0$, there
		is $\alpha_{i,j}\in \Q_{>0}$ 
		such that
		\[
		\Tr(x^{\sigma_i}x|_{V_{i,j}})
		=\alpha_{i,j} \Tr_{A_i/K}(x^{\sigma_i} x)
		\]
		for all $x\in A_i$. Provided this holds, the map
		$a\mapsto a^{(i)}:S_i\to A_i$
		defines an isometry
		from $q_{A,\sigma}|_{S_i}$ to $(\sum_j \alpha_{i,j})\cdot q_{A_i,\sigma_i}$,
		and the proposition follows.
		
		Fix $i$ and write $(B,\tau)=(A_i,\sigma_i)$. If $B$ is simple and $W$
		is a simple left $B$-module,
		then $V_{i,j}\cong W^m$ and ${}_B B\cong W^n$ for some $n,m$.
		Thus, for all $y\in B$,
		we have $\Tr(y|_{V_{i,j}})=\Tr(y|_{W^m})=m\Tr(y|_W)=\frac{m}{n}\Tr(y|_{W^n})=\frac{m}{n}\Tr_{B/K}(y)$,
		and we can take $\alpha_{i,j}=\frac{m}{n}$.
		
		If $B$ is not simple, then we may assume
		that $B=C\times C^\op$, where $C$
		a simple $K$-algebra, and $\tau$ is given by $(y,z^\op)\mapsto (z,y^\op)$.
		Let $U$ denote a simple left $C$-module and let $W$ denote a simple left $C^\op$-module.
		Then there is $n\in\N$
		such that   ${}_BB\cong U^n\oplus W^n$.
		Writing $x=(y,z^\op)\in C\times C^\op$,
		it is easy to   check that $\Tr_{B/K}(x^\tau x)=2\Tr_{C/K}(yz)=2n\Tr(yz|_U)=2n\Tr(x^\tau x|_U)$,
		and similarly, $\Tr_{B/K}(x^\tau x)=2n \Tr(x^\tau x|_W)$.
		Now, as in the previous paragraph,
		we see that $\Tr(x^\tau x|_{V_{i,j}})=\frac{m}{2n}\Tr_{B/K}(x^\tau x)$
		with $m=\mathrm{length}_B V_{i,j}$.
		This completes the proof.
	\end{proof}
	
	\begin{lem}\label{LM:reduction-mod-nilpotent}
		Let $(R,\sigma)$ be a ring with involution such that $2\in\units{R}$,
		let $J\idealof R$ be a nilpotent ideal such
		that $J^\sigma=J$, and let $\quo{\sigma}$
		denote the involution $r+J\mapsto r^\sigma+J:R/J\to R/J$.
		Then every idempotent $\veps\in R/J$ satisfying
		$\veps^{\quo{\sigma}}+\veps=1$ is the image of
		an idempotent $e\in R$
		satisfying $e^\sigma+e=1$. In particular,
		$\sigma$ is hyperbolic if and only if $\quo{\sigma}$
		is hyperbolic.
	\end{lem}

	\begin{proof}
		This is a special case of \cite[Corollary~4.9.16]{First_2012_PhD}.
		We recall and streamline the proof for the sake of completeness.
		
		It is enough to consider the case $J^2=0$.
		It is well-known that $\veps$ can be lifted
		to an idempotent $f\in R$.
		Write $a=\frac{1}{2}ff^\sigma $ and $b=\frac{1}{2}f^\sigma f$.
		Then $a,b\in J$ (because $\veps\veps^{\quo{\sigma}}=\veps^{\quo{\sigma}}\veps=0$), 
		$a^\sigma=a$, $b^\sigma=b$
		and   $a^2= b^2= ab=ba=0$
		(because $J^2=0$). In addition, $af=af^\sigma f=ab=0$,
		and similarly, $f^\sigma a=fb=b f^\sigma=0$.
		Let $e=f-a-b$.  The previous identities imply
		readily that $e^2=e$ and $ee^\sigma=e^\sigma e =0$.
		Thus, $e$ is an idempotent mapping onto $\veps$
		and $e+e^\sigma$ is an idempotent
		mapping onto $1+J$.
		Since $J$ is nilpotnet, we must have $e+e^\sigma=1$.
	\end{proof}
	
	\begin{lem}\label{LM:trace-of-inv-in-real-closed}
		Theorem~\ref{TH:trace-general} holds when $K$ is real closed.
		In fact, when the conditions hold, one can take $n=2$ in  (a).
	\end{lem}

	\begin{proof}
		Let $(\quo{A},\quo{\sigma})$ and $(A_i,\sigma_i)_{i=1}^t$
		be as in Lemma~\ref{LM:trace-description}.
		By part (ii) of that lemma,
		we have $\sgn q_{A,\sigma}=\sum_i \sgn q_{A_i,\sigma_i}$.
		Since $\sgn q_{A_i,\sigma_i}\geq 0$
		for all $i$ (Lemma~\ref{LM:the-real-simple-case}(i)),
		$\sgn q_{A,\sigma}=0$ if and only
		if $\sgn q_{A_i,\sigma_i}=0$ for all $i$.
		By Lemma~\ref{LM:the-real-simple-case}(ii),
		this is equivalent to the hyperbolicity of $2\times (\quo{A},\quo{\sigma})$,
		which is in turn equivalent to $2\times (A,\sigma)$ being hyperbolic, by
		Lemma~\ref{LM:reduction-mod-nilpotent}.
	\end{proof}
	
	\begin{lem}\label{LM:no-quad-extensions}
		Let $(A,\sigma)$ be a $K$-algebra with involution
		and let $\alpha\in \units{K}$ be an element such
		$\alpha$ and $-\alpha$ are not squares in $K$.
		If both 
		$\sigma_{K[\sqrt{\alpha}]}$
		and 
		$\sigma_{K[\sqrt{-\alpha}]}$
		are hyperbolic, then so is 
		$2\times \sigma$.
	\end{lem}
	
	\begin{proof}
		The case where $(A,\sigma)$ is central simple is a result of Lewis and Unger,
		see
		\cite[p.~475]{Lewis_2003_local_global_princ_alg_with_inv}
		or
		\cite[Lemma~6.1]{Becher_2018_weakly_hyerbolic_involutions}.
		We will derive the general case from their result.
		
		Let $(\quo{A},\quo{\sigma})$ and $(A_i,\sigma_i)_{i=1}^t$
		be as in Lemma~\ref{LM:trace-description}.
		By Lemma~\ref{LM:reduction-mod-nilpotent}, it is enough
		to prove the lemma for $(\quo{A},\quo{\sigma})$,
		which in turn amounts to proving it for each factor $(A_i,\sigma_i)$.
		We may therefore assume that $(A,\sigma)$ is simple.
		
		Write $F=\Cent(A)^{\{\sigma\}}$ and recall that $F$ is a field. If $\alpha$ is a square in $F$,
		then $F\otimes_K K[\sqrt{\alpha}]\cong F\times F$.
		As a result,  $(A_{K[\sqrt{\alpha}]},\sigma_{K[\sqrt{\alpha}]})\cong (A,\sigma)\times (A,\sigma)$,
		so $(A,\sigma)$  is hyperbolic by our assumptions.
		
		Similarly, $(A,\sigma)$ is hyperbolic if $-\alpha$ is a square
		in $F$.
		
		Finally, if neither $\alpha$ nor $-\alpha$ are squares in $F$,
		then we may regard $(A,\sigma)$ as a central simple
		$F$-algebra and finish by the result mentioned at the beginning.
	\end{proof}
	
	\begin{prp}[Bayer-Fluckiger, Lenstra]
		\label{PR:Bayer-Lenstra}
		Let $(A,\sigma)$ be a $K$-algebra with involution
		and let $L/K$ be a finite odd-degree field extension.
		Then $(A_L,\sigma_L)$ is hyperbolic if and only if $(A,\sigma)$ is hyperbolic.
	\end{prp}
	
	\begin{proof}
		By Proposition~\ref{PR:hyp-inv-vs-hyp-form}, it is enough to prove
		the corresponding statement for unimodular hermitian forms
		over $(A,\sigma)$. Write $W(A,\sigma)$
		for the Witt group of $1$-hermitisn forms over $(A,\sigma)$.
		By a result of
		Bayer-Fluckiger
		and Lenstra  \cite[Proposition~1.2]{Bayer_1990_odd_degree_extensions},
		the restriction map $W(A,\sigma)\to W(A_L,\sigma_L)$
		is injective. By \cite[Proposition~5.12]{Bayer_2014_hermitian_categories} (for instance),
		every hermitian form representing the trivial class in $W(A,\sigma)$
		is hyperbolic, so we are done.
	\end{proof}

	Now we can  prove Theorem~\ref{TH:trace-general}.
	
	\begin{proof}
		To see that (a)$\implies$(b), take an ordering $P$
		of $K$ and apply Lemma~\ref{LM:trace-of-inv-in-real-closed} after base changing
		from $K$ to $K_P$.		
 		We turn to show that   (b)$\implies$(a), and moreover, that $n$ in condition (a) can be taken
 		to be a power of $2$. By Corollary~\ref{CR:gcd-for-algebras},
 		this will imply that the minimal possible $n$ for which (a) holds is a power of $2$.
 		
		Let $\quo{K}$ be an algebraic closure of $K$ and suppose,
		for the sake of contradiction, that there
		is no $k\in \N\cup\{0\}$ such that $2^k\times(A,\sigma)$
		is hyperbolic. By Zorn's lemma,
		there exists a   $K$-subfield $L\subseteq \quo{K}$
		which is maximal relative to the property
		that  $2^k\times (A_L,\sigma_L)$ is not hyperbolic for all $k$.
		Proposition~\ref{PR:Bayer-Lenstra} implies that $L$ has no proper odd-degree
		finite extensions, and by Lemma~\ref{LM:no-quad-extensions},
		for every $\alpha\in \units{K}$, at least one of $\alpha$, $-\alpha$ is a square.
		In addition, $-1$ is not a sum of squares in $L$,
		otherwise there exists $k\in\N$ such that $(\nMat{L}{2^k},\trans)$,
		and hence $2^k\times \sigma_L$, is hyperbolic \cite[Proposition~6.2]{Becher_2018_weakly_hyerbolic_involutions}.
		We conclude that $L$ is real closed, but this contradicts
		Lemma~\ref{LM:trace-of-inv-in-real-closed}.
	\end{proof}
	
	\begin{remark}
		The fact that the minimal $n$ for which condition
		(a) in Theorem~\ref{TH:trace-general} holds is a power
		of $2$ can also be derived from 
		a theorem of Bayer-Fluckiger,
		Parimala and Serre \cite[Theorem~3.1.1]{Bayer_2013_Hasse_principle_for_G_traces},
		by arguing as in the proof of Corollary~\ref{CR:gcd-for-algebras}.
	\end{remark}

\section{Proof of Theorem~\ref{TH:main-I}}
\label{sec:systems-main}

	We  use Theorem~\ref{TH:trace-general} to prove
	Theorem~\ref{TH:main-I}.

	\begin{proof}[Proof of Theorem~\ref{TH:main-I}]
		Recall that we are given a $K$-vector space  with a system
		of forms $(V,\{q_i\}_{i\in I})$  and $(A,\sigma)$ is its algebra of adjoints together
		with the canonical involution.
	
		The theorem will follow from  Theorem~\ref{TH:trace-general}
		if we verify the following two facts:
		\begin{enumerate}[label=(\roman*)]
			\item The algebra of adjoints  of $n\times\{q_i\}_{i\in I}$
			with its canonical involution
			is isomorphic to $n\times (A,\sigma)$.	
			\item $\{q_i\}_{i\in I}$ is hyperbolic if and only if $(A,\sigma)$
			is hyperbolic.
		\end{enumerate}
		The first statement is straightforward. 
		As for the second,
		if $e=(\phi,\psi^\op)\in A$ is an idempotent
		such that $e+e^\sigma=1$,
		then $\phi$ and $\psi$ are idempotents in $\End(V)$
		such that $\phi+\psi=\id_V$. Thus, $V=\im \phi\oplus \im \psi$.
		Since for all $x,y\in V$,
		we have
		$q_i(\psi x,\psi y)=q_i(x,\phi \psi y)=0$
		and $q_i(\phi x,\phi y)=q_i(\psi \phi x,y)=0$,
		this means that $\{q_i\}_{i\in I}$ is hyperbolic.
		Conversely, if $V=U\oplus W$
		and each $q_i$ vanishes on both $U$ and $W$,
		then $e:=(\id_U\oplus 0_W, (0_U\oplus \id_W)^\op)$
		is easily seen to be an idempotent in $A$
		satisfying $e+e^\sigma=1$.
	\end{proof}	
	
	Theorem~\ref{TH:main-I} can be regarded
	as a generalization of Pfister's local-global principle by means of part (i)
	of the following proposition,
	which is well-known when $q$ is unimodular.
	
	\begin{prp}\label{PR:adj-of-single-form}
		Let $(V,q)$ be a quadratic space and let
		$(A,\sigma)$ denote the   algebra-with-involution of adjoints  
		of $\{q\}$. Then:
		\begin{enumerate}[label=(\roman*)]
			\item $\sgn q_{A,\sigma}=(\sgn q)^2$.
			\item $\Cad  \{q\} =K q$.
		\end{enumerate}
	\end{prp}
	
	\begin{proof}
		We may assume without loss of generality
		that $q=\langle\alpha_1,\dots,\alpha_m\rangle\oplus (n\times \langle 0\rangle)$
		with $\alpha_1,\dots,\alpha_m\in \units{K}$.
		Writing $g:=\mathrm{diag}(\alpha_1,\dots,\alpha_m)$
		and identifying $\End(V)=\End(K^{n+m})$
		with $\nMat{K}{n+m}$,
		straightforward computation now shows that
		\[
		A=\{([\begin{smallmatrix}
		a & 0 \\
		x & y
		\end{smallmatrix}],
		[\begin{smallmatrix}
		g^{-1}a^\trans g & 0 \\
		z & w		
		\end{smallmatrix}]^\op)
		\where
		a\in \nMat{K}{m},\, x,z\in \nMat{K}{n\times m},\, y,w\in\nMat{K}{n\times n} 
		\}
		\]
		
		(i) 		
		It is routine to check  
		that $q_{A,\sigma}\cong [\bigoplus_{i,j\in\{1,\dots,m\}}(\alpha_i\alpha_j^{-1})]\oplus
		[n^2\times \langle 1,-1\rangle]\oplus [2nm\times \langle 0\rangle]$.
		Thus, for every
		ordering $P$ of $K$,
		we have 
		\[\sgn_P q_{A,\sigma}=\sum_{i,j\in\{1,\dots,m\}}\sgn_P(\alpha_i)\sgn_P(\alpha_j)=
		\bigg(\sum_{i=1}^m\sgn(\alpha_i)\bigg)^2= (\sgn_P q)^2 .\]
		
		(ii) Let $w\in \Cad \{q\} $ and let $[\begin{smallmatrix}
		a & b \\
		c & d
		\end{smallmatrix}]$ denote
		its Gram matrix relative to the standard
		basis of $K^{m+n}$, where $a\in \nMat{K}{m}$.
		By the description of $A$ given above,
		we have $
		[\begin{smallmatrix}
		0 & 0 \\
		0 & I_n
		\end{smallmatrix}]		
		[\begin{smallmatrix}
		a & b \\
		c & d
		\end{smallmatrix}]
		=
		[\begin{smallmatrix}
		a & b \\
		c & d
		\end{smallmatrix}]
		[\begin{smallmatrix}
		0 & 0 \\
		0 & 0
		\end{smallmatrix}]$,
		$
		[\begin{smallmatrix}
		0 & 0 \\
		0 & 0
		\end{smallmatrix}]		
		[\begin{smallmatrix}
		a & b \\
		c & d
		\end{smallmatrix}]
		=
		[\begin{smallmatrix}
		a & b \\
		c & d
		\end{smallmatrix}]
		[\begin{smallmatrix}
		0 & 0 \\
		0 & I_n
		\end{smallmatrix}]$
		and
		$
		[\begin{smallmatrix}
		g^{-1}x^\trans g & 0 \\
		0 & 0
		\end{smallmatrix}]^\trans		
		[\begin{smallmatrix}
		a & b \\
		c & d
		\end{smallmatrix}]
		=
		[\begin{smallmatrix}
		a & b \\
		c & d
		\end{smallmatrix}]
		[\begin{smallmatrix}
		x & 0 \\
		0 & 0
		\end{smallmatrix}]$
		for all $x\in \nMat{K}{m}$.
		These equalities
		imply that $b=0$, $c=0$, $d=0$
		and $g^{-1}a\in \Cent(\nMat{K}{m})$.
		Consequently, $w=\alpha q$ for some $\alpha\in K$.
	\end{proof}

\section{Hermitian Categories}
\label{sec:herm-cat}

	We recall some facts about hermitian categories
	that will be needed for the proof of Theorem~\ref{TH:main-II} in the
	next section. 
	We refer the reader to
	\cite[\S2]{Bayer_2014_hermitian_categories} for the relevant definitions,
	and to \cite[\S7]{Scharlau_1985_quadratic_and_hermitian_forms}
	or \cite[Chapter~III]{Knus_1991_quadratic_hermitian_forms}
	for an extensive treatment.

	All categories are tacitly assumed
	to be skeletally small. The composition symbol ``$\circ$''
	will often be suppressed in formulas.

\medskip

	Let $\catC=(\catC,*,\omega)$ be a hermitian
	category.
	The category of unimodular $1$-hermitian spaces over $\catC$
	is denoted $\Herm{\catC}$.

	Given a set $I$, an $I$-indexed
	system of hermitian forms over $\catC$ is a pair  $(C,\{h_i\}_{i\in I})$
	consisting of an object $C\in \catC$
	and a collection $\{h_i\}_{i\in I}$ of (possibly non-unimodular) $1$-hermitian forms
	on $C$. An isometry
	from $(C,\{h_i\}_{i\in I})$ to another
	$I$-indexed system
	$(C',\{h'_i\}_{i\in I})$
	is an isomorphism $f:C\to C'$
	such that $f^*h'_if=h_i$ for all $i\in I$.
	The category of $I$-indexed
	systems of hermitian forms with isometries as morphisms 
	is denoted
	$\Sys_I(\catC)$.

	Define 
	the category $\tAr{\catC}{I}$ as follows:
	\begin{itemize}
		\item 
		Objects are triples $(U,V,\{f_i\}_{i\in I})$
		with $U,V\in \catC$
		and $\{f_i\}_{i\in I}\subseteq\Hom_{\catC}(U,V^*)$.
		\item 
		Morphisms from $(U,V,\{f_i\}_{i\in I})$
		to $(U',V',\{f'_i\}_{i\in I})$
		are formal
		symbols   $(\phi,\psi^\op)$ such
		that $\phi\in \Hom_{\catC}(U,U')$, $\psi\in\Hom_{\catC}(V',V)$
		and $\psi^*f_i=f_i\phi$ for all $i\in I$.
		\item
		Composition is defined by $(\phi,\psi^\op)\circ (\phi',\psi'^\op):=
		(\phi  \phi',(\psi'  \psi)^\op)$.
	\end{itemize}
	We call  $\tAr{\catC}{I}$ the category
	of \emph{twisted arrows} over $\catC$
	and make it into a hermitian category
	by setting 
	\begin{align*}
	&(U,V,\{f_i\}_{i\in I})^*=(V,U,\{f_i^*\circ \omega_U\}_{i\in I}),\\
	&(\phi,\psi^\op)^*=(\psi,\phi^\op),   \\
	&\omega_{(U,V,\{f_i\})}=(\id_U,\id_V^\op).
	\end{align*}
	Note that  if $h=(\phi,\psi^\op)$
	is a hermitian form on $Z\in \tAr{\catC}{I}$,
	then $\phi=\psi$, because $h=h^*\omega_Z$.
	
	We alert the reader
	that  $\tAr{\catC}{I}$ is not the category
	of twisted \emph{double} $I$-arrows defined in \cite[\S4]{Bayer_2014_hermitian_categories}
	and denoted
	$\tAr{\catC}{2I}$. Rather, $(U,V,\{f_i\})\mapsto (U,V,\{f_i\},\{f_i\})$
	identifies $\tAr{\catC}{I}$ as a full hermitian subcategory
	of $\tAr{\catC}{2I}$.

\medskip
	
	The following theorem is a variation of \cite[Theorem~4.1]{Bayer_2014_hermitian_categories}
	
	\begin{thm}\label{TH:equiv-syst-to-single}
		Define   $F: \Sys_I(\catC)\to \Herm{\tAr{\catC}{I}}$ 
		and $G:\Herm{\tAr{\catC}{I}}\to \Sys_I(\catC)$  by
		\begin{align*}
		&F(C,\{h_i\})=((C,C,\{h_i\}),(\id_C,\id_C^\op)),&  &F(\phi)=(\phi,(\phi^{-1})^\op).\\
		&G((U,V,\{f_i\}),(\alpha,\alpha^\op))=(U,\{\alpha^* f_i\}_{i\in I}), &
		&G(\phi,\psi^\op)=\phi.
		\end{align*}
		Then $F$ and $G$ are well-defined
		functors which are mutually inverse.
		Moreover $F$ and $G$ respect orthogonal sums
		and preserve hyperbolicity.
	\end{thm}

	\begin{proof}
		We already
		observed that $\tAr{\catC}{I}$
		is a full subcategory of $\tAr{\catC}{2I}$.
		With this observation at hand,
		the proof is the same as the
		proof of   \cite[Theorem~4.1, Proposition~4.2]{Bayer_2014_hermitian_categories}.
	\end{proof}	
	
	\begin{remark}\label{RM:adjoint-ring-desc}
		Take $\catC$ to be the category
		of finite-dimensional $K$-vector spaces
		with the usual duality
		$V^*=\Hom_K(V,K)$,
		let $(V,\{q_i\}_{i\in I})$ be a system of quadratic
		forms over $\catC$, 
		and let $(Z,h)=F(V,\{q_i\}_{i\in I})$.
		Then $\End_{\tAr{\catC}{I}}(Z)$
		is precisely the algebra of adjoints  $A=A(\{q_i\}_{i\in I})$,
		and the canonical involution $\sigma:A\to A$
		coincides with the involution $f\mapsto h^{-1}f^*h$
		on $\End(Z)$. The
		information that $(A,\sigma)$ carries on the system $\{q_i\}_{i\in I}$
		is therefore a manifestation of Theorem~\ref{TH:equiv-syst-to-single}.	
	\end{remark}
	
	We call the hermitian category $\catC$ a \emph{finite hermitian	$K$-category}
	if the $\Hom$-groups in $\catC$ are finite dimensional $K$-vector spaces,
	composition is $K$-bilinear and $*:\catC\to \catC$ is $K$-linear on $\Hom$-groups.
	Recall also that $\catC$ is \emph{pseudo-abelian} if every
	idempotent morphism has a kernel.
	For example, both properties
	are satisfied by the category of finite-dimensional
	$K$-vector spaces with the usual duality
	$V^*=\Hom_K(V,K)$. In addition, if $\catC$
	satisfies either property, then so does $\tAr{\catC}{I}$.

	Suppose henceforth that $\catC$ is a semi-abelian finite hermitian $K$-category,
	and fix a collection  $\catZ$ of objects
	in $\catC$
	such that every indecomposable object $C\in \catC$
	satisfies $C\cong Z$ or $C\cong Z^*$ for unique $Z\in \catZ$.
	Given $Z\in\catZ$, a hermitian space $(C,h)\in \Herm{\catC}$
	is said to be of type $Z$ if $C$ is isomorphic
	to  a summand of $(Z\oplus Z^*)^n$ for some $n\in\N$.
	The following theorems are due
	to Quebbemann, Scharlau and Schulte  \cite[Theorems~3.2, 3.3]{Quebbemann_1979_hermitian_categories}.
	
	\begin{thm}\label{TH:QSSi}
		Every $(C,h)\in \Herm{\catC}$
		admits a factorization
		\[
		(C,h)\in \bigoplus_{Z\in\catZ}(C_Z,h_Z)
		\]
		in which $(C_Z,h_Z)$ 
		is a unimodular hermitian space of type $Z$
		and $C_Z=0$ for all but finitely many $Z\in \catZ$.
		The factors $(C_Z,h_Z)$ are uniquely determined up to isometry. 
	\end{thm}
	
	\begin{thm}\label{TH:QSSii}
		Let $Z\in \catZ$, let $\veps\in \{\pm 1\}$
		and let $h:Z\to Z$ be a unimodular $\veps$-hermitian form.
		Write $E=\End_{\catC}(Z)$ and $\quo{E}=E/\Jac E$,
		let $\sigma:E\to E$ be given by $\phi^\sigma =h^{-1}\phi^*h$
		and let $\quo{\sigma}:\quo{E}\to \quo{E}$
		be given by $(\quo{\phi})^{\quo{\sigma}}=\quo{\phi^\sigma}$.
		Then there is a functor $T$ from the category
		of unimodular $1$-hermitian spaces of type $Z$ over $\catC$
		to the category
		of unimodular $\veps$-hermitian spaces over $(\quo{E},\quo{\sigma})$
		having the following properties:
		\begin{enumerate}[label=(\roman*)]
			\item $T$ induces   a bijection on   isomorphism classes.
			\item $T$ respects orthogonal sums and hyperbolicity.
			\item $T$ maps $1$-hermitian spaces over $\catC$
			with underlying object  $Z^n$ to
			$\veps$-hermitian spaces over $(\quo{E},\quo{\sigma})$
			with underlying module $\quo{E}^n$.
		\end{enumerate}
	\end{thm}

\section{Proof of Theorem~\ref{TH:main-II}}
\label{sec:pairs}

	We finally   prove Theorem~\ref{TH:main-II}.
	In fact, we shall prove a more general result. 
	To state it,   we introduce the following condition on the  field $K$:
	\begin{enumerate}[label=(E)]
		\item \label{item:condition-on-K} 
		For every ordering $P$ of $K$ and every two disjoint finite
		subsets $S,T\subseteq K_P$, there exists $f\in K[X]$
		such that $f(s)>_P0$ for all $s\in S$ and $f(t)<_P 0$ for all $t\in T$.
	\end{enumerate}
	Fields satisfying \ref{item:condition-on-K} include all real closed fields
	and fields  which are dense in their real closure relative
	to each of their orderings, 
	e.g.\ number fields. (Indeed,
	take $f$ to be an approximation of
	an appropriate  interpolation polynomial in $K_P[X]$.)
	We conjecture that \ref{item:condition-on-K} holds for all fields.
	
\medskip
	
	Theorem~\ref{TH:main-II} is a special case of:
	
	\begin{thm}\label{TH:answer-to-Qtwo}
		Suppose that $K$ satisfies condition \ref{item:condition-on-K},
		e.g.,  $K$ is  a number field or a real-closed field.
		Let $(V,\{q_i\}_{i=1,2})$ be a  $K$-vector space with a nonsingular
		pair of quadratic forms.
		Then the following conditions are equivalent:
		\begin{enumerate}[label=(\alph*)]
			\item $n\times \{q_i\}_{i=1,2}$ is hyperbolic form some $n\in \N$.
			\item Every   $q\in \Cad \{ q_1,q_2 \} $  
			has signature $0$.
		\end{enumerate}
	\end{thm}
	
	For the remainder of this section, 
	we adopt the convention of Section~\ref{sec:examples} in 	
	which quadratic forms on $K^n$ are identified with their Gram matrix relative
	to the standard basis.
	In addition, we  let
	$\catC$ denote
	the category whose   objects  are
	the $K$-vector spaces $\{K^0,K^1,K^2,\dots \}$,
	its morphisms are given by $\Hom(K^n,K^m)=\nMat{K}{m\times n}$, and its composition is matrix
	product. We make $\catC$
	into a hermitian category
	by letting  $*$ fix all objects
	and act  as  the matrix transpose on morphisms, and setting $\omega=\id$.
	Write $\Sys_2(K)=\Sys_{\{1,2\}}(\catC)$. Then, under our conventions,
	objects of
	$\Sys_2(K)$ can be regarded as vector spaces
	with a pair  of quadratic forms in the sense of Section~\ref{sec:preliminaries}, 
	and all such a pairs are obtained in this manner,
	up to isomorphism.

	Denote $\tAr{\catC}{\{1,2\}}$ as $\tAr{K}{2}$
	and write the objects $(U,V,\{f_i\}_{i=1,2})$ of $ \tAr{K}{2}$ as quadruples
	$(U,f_1,f_2,V)$.
	The assignment $(U,f_1,f_2,V)\mapsto (U,f_1,f_2,V^*)$
	and $(\phi,\psi^\op)\mapsto (\phi,\psi^*)$
	defines an equivalence between
	$\tAr{K}{2}$ and the category of \emph{Kronecker modules},
	i.e., pairs of vector spaces with a pair of linear maps from 
	the first space to the second space.
	
	If $L$ is a $K$-field, then we
	have evident base change functors $\Sys_2(K)\to \Sys_2(L)$
	and $\tAr{K}{2}\to \tAr{L}{2}$.
	It    is routine to check that the equivalence
	of Theorem~\ref{TH:equiv-syst-to-single} is compatible these functors
	(see \cite[\S3D, Remark~2.2]{Bayer_2014_hermitian_categories} for a generalization).

	Given $n\in \N$ and $\alpha\in F$,
	we define the following    $n\times n$ matrices 
	\[
	J_n(\alpha)=\left[\begin{matrix}
	\alpha & 1 &  & \\
	& \ddots & \ddots & \\
	& & \alpha & 1 \\
	& & & \alpha
	\end{matrix}\right],
	\quad
	S_n(\alpha)=\left[\begin{matrix}
	& &  & \alpha \\
	&   & \alpha & 1 \\
	  & \uddots & \uddots \\
	\alpha & 1
	\end{matrix}\right],
	\quad
	T_n =\left[\begin{matrix}
	& &    1 \\
	& \uddots &   \\
	1 &  & 
	\end{matrix}\right],
	\]
	and the    object 
	\[
	Z_n(\alpha) =
	(K^n,T_n,S_n(\alpha),K^n)\in \tAr{K}{2},
	\]
	which is isomorphic to $(K^n,I_n,J_n(\alpha),K^n)$ via $(I_n,T_n^\op)$.
	In addition, given $A\in \nMat{K}{n}$, we write
	\[
	Z(A)=(K^n,I_n,A,K^n)\in\tAr{K}{2}.
	\]
	It is easy to see that $Z(A)\cong Z(A')$
	if and only if $A$ and $A'$ are conjugate.
	Furthermore, $Z(A)$ is indecomposable
	in $\tAr{\catC}{I}$
	if and only if $K^n$ is an indecomposable
	$K[A]$-module. In this case, the characteristic
	polynomial of $A$ coincides with its minimal polynomial
	and is a prime power, and every matrix which commutes with $A$   belongs to $K[A]$.
	
	\begin{lem}\label{LM:good-decomp-for-main-II}
		Let $(K^n,\{q_i\}_{i=1,2})\in \Sys_2(K)$ 
		and let $C=F(K^n,\{q_i\}_{i=1,2})=(K^n,q_1,q_2,K^n)$ (see Theorem~\ref{TH:equiv-syst-to-single}).
		If $q_1$ is unimodular, then $C$ is a direct sum  of indecomposable
		objects of the form $Z(A)$.
	\end{lem}
	
	\begin{proof}	
		We observed above that $\tAr{K}{2}$
		is equivalent to the category of Kronecker modules over $K$.
		The indecomposable Kronecker modules were
		classified in the 19th century by Kronecker, see \cite[p.~69]{Gabriel_1974_degenerate_bilinear_forms},
		for instance.
		If $Z=(U,f,g,V)$ is an idecomposable summand of   $C$,
		then $f$ must be invertible, in which case it follows
		easily from the classification
		that $Z\cong Z(A)$ for some $A\in\nMat{K}{n}$, $n\in\N$.
	\end{proof}
	
	\begin{lem}\label{LM:Z-properties}
		Let $Z=Z(A)$ be an idecomposable object in $\tAr{K}{2}$.
		Then:
		\begin{enumerate}[label=(\roman*)]
			\item There exists an isomorphism $h:Z\to Z^*$
			such that $h=h^*$.
			\item For every $h$ as in (i),
			the involution $f\mapsto h^{-1}f^*h:\End(Z)\to \End(Z)$
			is the identity.
			\item Every unimodular hermitian space of type $Z$ over $\tAr{K}{2}$
			is the orthogonal sum of hermitian spaces with underlying object isomorphic
			to $Z$.
		\end{enumerate}
	\end{lem}
	
	\begin{proof}
		(i)
		Saying   that $h=(\phi,\phi^\op):Z\to Z^*$
		is an isomorphism
		amounts to saying that $\phi$ is a symmetric
		invertible $n\times n $ matrix 
		such that $\phi A\phi^{-1}=A^\trans$. 
		It is well-known that such $\phi$ exists, 
		see \cite[Theorem~66]{Kaplansky_2003_linear_algebra_geometry_reprint},
		for instance.
				
		(ii)
		Write   $h=(\phi,\phi^\op)$
		with $\phi\in\nGL{K}{n}$.
		Then $\phi=\phi^\trans$ and
		$\phi^\trans A =A^\trans  \phi$.
		Let    $f=(\xi,\psi^\op)\in \End(Z)$. Then  
		$\psi^\trans=\xi$
		and $\psi^\trans A=A\xi$,
		hence
		$\xi A=A\xi $.
		Since $K^n$ is an indecomposable
		$K[A]$-module,
		$\xi=p(A)$ for some $p\in K[X]$,
		so $\phi^{-1} \xi\phi=p(\phi^{-1} A^\trans \phi)=p(A)=\xi$.
		This implies readily that $h^{-1}f^*h=f$.
		
		(iii) Write $E=\End(Z)$. By (i) and Theorem~\ref{TH:QSSii},
		we reduce into showing that
		every unimodular $1$-hermitian form over $(\quo{E},\quo{\sigma})$
		is diagonalizable. 
		By Fitting's lemma, $\quo{E}$ is a field,
		and by (ii), $\quo{\sigma}=\id$.
		Since every
		quadratic form over a field of characteristic not $2$ is diagonalizable, we are done.
	\end{proof}

	\begin{lem}\label{LM:Z-over-R-structure}
		Let $P$ be an ordering of $K$
		and let $R=K_P$.
		Let $A\in \nMat{K}{n}$
		be a matrix such that $K^n$ is an indecomposable
		$K[A]$-module, let $f$ denote its characteristic
		polynomial and let $\alpha_1,\dots,\alpha_t$ be
		the roots of $f$ in $R$.
		Let $h:Z(A)\to Z(A)$ be a unimodular hermitian form.
		Then there exist $r\in \N$ and  a decomposition
		\[
		(Z(A),h)_R=(Z_0 ,h_0)\oplus \bigoplus_{i=1}^t (Z_r(\alpha_i),h_i)
		\]
		in $\Herm{\tAr{R}{2}}$ such that $h_0\oplus h_0$ is hyperbolic. 
	\end{lem}
	
	\begin{proof}
		There is a monic prime polynomial $p\in K[X]$
		and $r\in \N$
		such that $f=p^r$. Since $R$ is real-closed,
		there are distinct prime degree-$2$
		polynomials
		$q_1,\dots,q_s\in R[X]$
		such that $p =(X-\alpha_1)\cdots(X-\alpha_t)q_1\cdots q_s$.
		For each $j\in\{1,\dots,s\}$,
		choose $B_j\in \nMat{R}{2r}$ with minimal polynomial
		$q_j^r$.
		Then $A$ is conjugate
		to $J_{r }(\alpha_1)\oplus\dots\oplus J_{r }(\alpha_t)\oplus B_1\oplus \dots\oplus B_s$,
		and hence $Z(A)_R\cong Z_r(\alpha_1)\oplus \dots\oplus Z_r(\alpha_t)\oplus
		Z(B_1)\oplus\dots\oplus Z(B_s)$.
		By Lemma~\ref{LM:Z-properties}(i), each of
		the objects $Z_r(\alpha_1),\dots,Z_r(\alpha_t),Z(B_1),\dots,Z(B_s)$
		is isomorphic to its dual.
		Furthermore, these objects
		are pairwise non-isomorphic because  $J_{r }(\alpha_1),\dots, J_{r }(\alpha_t), B_1, \dots, B_s$
		are pairwise non-conjugate.
		Now, by 
		Theorem~\ref{TH:QSSi}, $(Z(A)_R,h_R)$
		factors as
		$[\bigoplus_{i=1}^t (Z_r(\alpha_i),h_i)]\oplus[\bigoplus_{j=1}^s (Z(B_j),h'_j)]$.
		We take $(Z_0,h_0)=\bigoplus_{j=1}^s (Z(B_j),h'_j)$.
		
		To complete the proof, we need to check that $h'_j\oplus h'_j$ is hyperbolic for
		all $j$.
		Write $E=\End(Z(B_j))$. It is easy
		to see that the assignment $(\phi,\psi^\op)\mapsto \phi$
		defines a $R$-algebra isomorphism between
		$E$
		and the centralizer of $B_j$ in $\nMat{R}{2r}$,
		which is just $R[B_j]$.
		Thus, $E\cong R[X]/(q_j^s)$.
		By Theorem~\ref{TH:QSSii} and Lemma~\ref{LM:Z-properties}, it
		is enough to
		show 
		that every $2$-dimensional unimodular quadratic form over $\quo{E}:=E/\Jac E\cong R[X]/(q_j)$
		is hyperbolic. This holds because  $\quo{E}$ is a quadratic extension of the
		real-closed field $R$, and hence algebraically closed.
	\end{proof}
	
	\begin{lem}\label{LM:quad-forms-on-Z-real-case}
		Assume that $K$ is real closed,
		let $\alpha\in K$,    
		let $h=(\phi,\phi^\op):Z_n(\alpha)\to Z_n(\alpha)^*$ be a unimodular
		hermitian form and write  
		$(K^n,\{q_i\}_{i=1,2})=G(Z,h)$ (see Theorem~\ref{TH:equiv-syst-to-single}).
		Then:
		\begin{enumerate}[label=(\roman*)]
			\item There exist   $\beta_1\in\units{K}$
			and $\beta_2,\dots,\beta_n\in K$
			such that
			\[
		\phi=\left[\begin{matrix}
		\beta_1 & \beta_2 & \cdots  & \beta_n \\
		& \ddots & \ddots & \vdots \\
		& & \beta_1 &  \beta_2\\
		& & & \beta_1
		\end{matrix}\right].
		\]
			\item $(Z,h)\cong (Z,(I_n,I_n^\op))$
		if $\beta_1>0$ and $(Z,h)\cong (Z,(-I_n,-I_n^\op))$
		if $\beta_1<0$.
		\end{enumerate}
	\end{lem}
	
	\begin{proof}
		(i) We have $\phi^\trans T_n=T_n\phi$
		and $\phi^\trans S_n(\alpha)=S_n(\alpha) \phi$,
		hence $J_n(\alpha)\phi=T_n^{-1}  S_n(\alpha) \phi=
		T_n^{-1} \phi^\trans S_n(\alpha)=\phi T_n^{-1} S_n(\alpha) =\phi J_n(\alpha)$,
		so $\phi$ commutes with $J_n(\alpha)$.
		Since $K^n$ is an indecomposable $K[J_n(\alpha)]$-module,
		this means that $\phi$ is a polynomial in $J_n(\alpha)$, and thus has
		the desired form.
		
		(ii) It is easy to see that there exists 
		$\eta\in \nGL{K}{n}$, taking the same form as $\phi$
		such that $\sgn(\beta_1)\eta^2=\phi$.
		It is routine to check
		that $\eta^\trans T_n=T_n\eta$ and $\eta^\trans S_n(\alpha)=S_n(\alpha)\eta$,
		from which it follows that $(\eta,\eta^\op)$
		is an isometry
		from $(Z_n(\alpha),(\phi,\phi^\op))$ to $(Z_n(\alpha),(\sgn(\beta_1)I_n,\sgn(\beta_1)I_n^\op))$.
	\end{proof}
	
	\begin{lem}\label{LM:J-properties}
		Let $(K^n,\{q_i\}_{i=1,2})\in\Sys_2(K)$ 
		and assume that $q_1$ is unimodular.
		Let $J=q_1^{-1}q_2$ and   $Z= (K^n,q_1,q_2,K^n)$. Then:
		\begin{enumerate}[label=(\roman*)]
			\item $(J,J^\op)$ is a central element of $\End (Z)$.
			\item For every $f\in K[X]$, the morphism
			$ q_1 f(J):K^n\to (K^n)^*$ in $\catC$ defines a quadratic form   in $\Cad \{q_1,q_2\} $.
			\item Continuing (ii), if $\{q_i\}_{i=1,2}$ is hyperbolic,
			then so is $q_1 f(J)$.
		\end{enumerate}
	\end{lem}
	
	\begin{proof}
		(i) One readily checks
		that $J^\trans q_1=q_1 J$
		and $J^\trans q_2=q_2 J$,
		hence $(J,J^\op)\in \End (Z)$.
		Let $(\phi,\psi^\op)\in \End_{\catC}(Z)$.
		Then $\psi^\trans q_1=q_1\phi$
		and $\psi^\trans q_2=q_2\phi$,
		from which it follows
		that
		$\phi J=\phi q_1^{-1}q_2=q_1^{-1}\psi^\trans q_2=q_1^{-1} q_2\phi=J\phi$
		and $\psi J=\psi q_1^{-1}q_2=q_1^{-1}\phi^\trans q_2=q_1^{-1}q_2\psi=J\psi$,
		so $(\phi,\psi^\op)$ and $(J,J^\op)$ commute.

		(ii) 
		Let $(\phi,\psi^\op)\in A(\{q_i\}_{i=1,2})=\End(Z)$
		(see Remark~\ref{RM:adjoint-ring-desc}).
		We need to check that
		$\psi^\trans q_1 f(J)=q_1 f(J) \phi$.
		We observed in the proof of (i) that $\phi J=J\phi$,
		so  
		$\psi^\trans q_1 f(J)=q_1 \phi f(J)=q_1 f(J) \phi$.

		(iii) This follows from (ii) and claim (ii) in the proof of Theorem~\ref{TH:main-I}.
	\end{proof}

	We are now ready to prove Theorem~\ref{TH:answer-to-Qtwo}, thus establishing 
	Theorem~\ref{TH:main-II}.

	\begin{proof}[Proof of Theorem~\ref{TH:answer-to-Qtwo}]
		By Theorem~\ref{TH:main-I}, the statement
		is vacuous if $K$ is not real. We may therefore assume
		that $K$ is real, and in particular infinite.	
		Now, by Lemma~\ref{LM:nonsingular-systems},
		$\Span_K\{q_1,q_2\}$ contains a unimodular form.
		Replacing $\{q_i\}_{i=1,2}$ with another pair spanning   
		$\Span_K\{q_1,q_2\}$, we may assume that $q_1$ is unimodular.
		The implication (a)$\implies$(b) 
		was explained in the introduction, so we turn to prove the converse.
		In fact, by Theorem~\ref{TH:main-I}, it is enough
		to check that for every ordering $P$ of $K$,
		the system $2\times \{(q_i)_{K_P}\}_{i=1,2}$ is hyperbolic.

		We may assume that $V=K^v$ for some $v$; recall that we treat  $q_1,q_2$
		as symmetric $v\times v$ matrices.
		Write $R=K_P$, $J=q_1^{-1}q_2$, $Z=(K^v,q_1,q_2,K^v)$, $h=(I_v,I_v^\op):Z\to Z^*$
		and note that $(Z,h)=F(V,\{q_i\}_{i=1,2})$ (see Theorem~\ref{TH:equiv-syst-to-single}).
		Let $\calP\subseteq K[X]$ denote the monic prime factors
		of the characteristic polynomial of $J$.
		For every $p\in \calP$ and $n\in\N$, let $A_{p,n}$ denote
		a fixed square matrix over $K$ with characteristic
		polynomial $p^n$. Then $Z(A_{p,n})\in \tAr{K}{2}$
		is indecomposable and isomorphic
		to its dual (Lemma~\ref{LM:Z-properties}(i)).

		Note that if $Z\cong (U,f,g,V)$ in $\tAr{K}{2}$,
		then $f^{-1}g$ is conjugate to $J$.
		Thus,  
		by Theorem~\ref{TH:QSSi} and Lemma~\ref{LM:good-decomp-for-main-II},
		we have a decomposition
		\begin{align}\label{EQ:decI}
		(Z,h)\cong \bigoplus_{p\in\calP}\bigoplus_{n \in\{1,\dots,  N\}} (Z_{p,n},h_{p,n}),
		\end{align}
		where $N$ is a sufficiently
		large integer and $Z_{p,n}$ is a direct sum of
		copies of $Z(A_{p,n})$.

		Fix some $p\in\calP$. We will prove that $2\times (h_{p,n})_R$ is hyperbolic for all
		$p$ and $n$ by a decreasing inducting on $n$.
		Thanks to Theorem~\ref{TH:equiv-syst-to-single}, this will finish the proof.
		The case     $n=N+1$ holds vacuously, so
		assume that the claim has been established for all $n>k$ for some $k\leq N$.

		Write $W=Z(A_{p,k})$  and
		let $\alpha_1,\dots,\alpha_t$ denote the roots of $p$ in $R$.
		Thanks to Lemma~\ref{LM:Z-properties}(iii),
		there is a decomposition
		\begin{align}\label{EQ:decIII}
		(Z_{p,k},h_{p,k})=\bigoplus_{j=1}^s (W,w_j),
		\end{align}
		and by Lemma~\ref{LM:Z-over-R-structure},
		we further have
		\begin{align}\label{EQ:decIV}
		(W,w_j)_R\cong   (W'_j,w'_j)\oplus \bigoplus_{i=1}^t (Z_{k}(\alpha_i),w_{ji})
		\end{align}
		with $2\times w'_j$ being hyperbolic.
		
		Fixing $i\in\{1,\dots,t\}$ and $j\in\{1,\dots,s\}$ and writing $w_{ji}=(\phi_{ji},\phi_{ji}^\op)$,
		Lemma~\ref{LM:quad-forms-on-Z-real-case}(i) implies
		that $\phi_{ji}$ is an upper-triangular matrix with constant diagonal;
		denote the scalar occurring on the diagonal of $\phi_{ji}$  by $\beta_{ji}$.
		Write
		\[r_i=\#\{j\in \{1,\dots,s\}\suchthat \beta_{ji}>0\}
		\qquad\text{and}
		\qquad
		r'_i=\#\{j\in \{1,\dots,s\}\suchthat \beta_{ji}<0\} .\]
		Then, by Lemma~\ref{LM:quad-forms-on-Z-real-case}(ii),
		\[
		(Z_{p,k},h_{p,k})_R\cong \bigoplus_{j=1}^s  
		(W'_j,w'_j)\oplus \bigoplus_{i=1}^t \bigg(Z_{k}(\alpha_i)^{r_i +r'_i},
		r_i\times (I_k,I_k^\op)\oplus r'_i\times (-I_k,-I_k^\op)\bigg).
		\]
		As a result, if $r_i=r'_i$ for all $i\in \{1,\dots,t\}$,
		then $2\times (h_{p,k})_R$ is hyperbolic.

		Let $g $ denote the product of all primes in $\calP-\{p\}$,
		and let $m\in \{0,\dots,t\}$.
		By virtue of condition \ref{item:condition-on-K},
		we can find 
		a polynomial $f_m\in K[X]$ such that for all $i\in\{1,\dots,t\}$,
		\[
		\veps_{m,i}:=f_m(\alpha_i)g^N(\alpha_i)\prod_{\ell\neq i}(\alpha_i-\alpha_\ell)<0 \quad
		\iff \quad m=i.
		\]		
		Let 
		$z_m=q_1p^{k-1}(J)g^N(J)f_m(J)$.
		By Lemma~\ref{LM:J-properties}(ii),
		$z_0,\dots,z_t$ are quadratic forms in $\Cad \{q_1,q_2\} $, so  they all have
		$P$-signature $0$ by assumption.
		
		Observe that $(V,\{q_i\}_{i=1,2})=GF(V,\{q_i\}_{i=1,2})=G(Z,h)$ (see Theorem~\ref{TH:equiv-syst-to-single}). 
		Since $G$ respects orthogonal sums,
		each of the morphisms $q_1,q_2, J,z_0,\dots,z_t$ in $\catC$ factors as a direct sum of components
		corresponding to the decomposition \eqref{EQ:decI}.
		Given $(q,n)\in \calP\times\{1,\dots,N\}$,
		write $z_{m,q,n}$
		for the component of $z_m$ corresponding to $Z_{q,n}$.
		By the definition of $G$ in Theorem~\ref{TH:equiv-syst-to-single},
		we have
		$z_{m,p,k}=(I ^\trans I )\cdot p^{k-1}(I ^{-1}A_{q,n})\cdot
		g^{N }(I ^{-1}A_{q,n})\cdot f(I^{-1} A_{q,n})=
		p^{k-1}( A_{q,n}) 
		g^{N }( A_{q,n})f( A_{q,n})$.
		From this we see that $z_{m,q,n}=0$ if $n<k$ or $q\neq p$.
		In addition, Lemma~\ref{LM:J-properties}(iii) and the induction
		hypothesis imply that $2\times z_{m,p,n}$ is hyperbolic
		whenever $n>k$.
		As a result, $\sgn_P z_{m,p,k}=\sgn_P z_m=0$.

		Write $b_m=(z_{m,p,k})_R$. Then $b_m $ isomorphic
		to an orthogonal sum of components corresponding to
		the decompositions~\eqref{EQ:decIII} and~\eqref{EQ:decIV};
		write $b_{m,j1},\dots,b_{m,jt}$  and $ b'_{m,j}$  for the components
		corresponding to $w_{j1},\dots,w_{jt}$ and $ w'_j$,
		respectively.
		As in the previous paragraph,
		$b'_{m,j}$ is hyperbolic,
		and  we have
		\begin{align*}
		b_{m,ji}&= \phi_{ji}^\trans 	T_n  \cdot f_m(T_k^{-1}S_k(\alpha_i)) \cdot 
		g^N(T_k^{-1}S_n(\alpha_i))\cdot p^{k-1}(T_k^{-1}S_k(\alpha_i))\\
		&=
		\phi_{ji}^\trans 	T_n  \cdot f_m(J_n(\alpha_i)) \cdot 
		g^N(J_k(\alpha_i))\cdot p^{k-1}(J_k(\alpha_i))=
		\begin{bmatrix}
		\ddots & & \vdots \\
		& 0 & 0 \\
		\cdots & 0 &  \beta_{ji}\veps_{m,i}
		\end{bmatrix}
		\end{align*}
		As a result, for $m\in \{1,\dots,t\}$, we have
		\[
		0=\sgn_Pz_{m,p,k}= \sgn b_m = \sum_{i,j} \sgn(\beta_{ji}\veps_{m,i})=-(r_m-r'_m)
		+\sum_{i\neq m} (r_i-r'_i),
		\]
		whereas for $m=0$, we get
		\[
		0=\sgn_Pz_{0,p,k}=\sgn b_0 = \sum_{i,j} \sgn(\beta_{ji}\veps_{0,i})= \sum_{i } (r_i-r'_i).
		\]
		Subtracting both equations gives $2(r_m-r'_m)=0$, so $r_m=r'_m$
		for all $m\in \{1,\dots,t\}$. This completes the proof.
	\end{proof}

\bibliographystyle{plain}
\bibliography{MyBib_18_05}

\end{document}